\documentclass[11pt,oneside,a4paper]{article}
\usepackage{syntonly}
\usepackage{geometry}
\usepackage{mathrsfs}
\usepackage{makeidx}
\makeindex
\usepackage{amsmath}
\usepackage{amssymb}
\usepackage{amsmath}
\usepackage{amsfonts}
\usepackage{amssymb}
\usepackage{esvect}
\usepackage{algorithm}
\usepackage{bm}
\usepackage{tikz}

\usepackage{cleveref}

\usepackage{graphicx}
\usepackage{epstopdf}

\newtheorem{theorem}{Theorem} 
\newtheorem{lemma}{Lemma}

\newtheorem{corollary}{Corollary}

\newenvironment{proof}[1][Proof]{\begin{trivlist}
\item[\hskip \labelsep {\bfseries #1}]}{\end{trivlist}}

\newcommand{\qed}{\nobreak \ifvmode \relax \else
      \ifdim\lastskip<1.5em \hskip-\lastskip
      \hskip1.5em plus0em minus0.5em \fi \nobreak
      \vrule height0.30em width0.4em depth0.25em\fi}
      
\newcommand*{\sign}{\mathop{\mathrm{sign}}\nolimits}

 \author{Safari Mukeru and Mmboniseni P. Mulaudzi\\
\footnotesize{\em Department of Decision Sciences}\\
\footnotesize{University of South Africa, P. O. Box 392, Pretoria, 0003. South Africa}\\
\footnotesize{e-mails: mukers@unisa.ac.za, mulaump@unisa.ac.za}}

\title{{Zeros of Gaussian power series, Hardy spaces and \\ determinantal point processes}}

\date{}

\begin{document}

\maketitle

\begin{center} \texttt{This paper is dedicated to the memory of President John Pombe Joseph Magufuli.} \end{center}

\pagenumbering{arabic}

\begin{abstract}
Given a sequence $(\xi_n)$ of standard i.i.d complex Gaussian random variables, Peres and Vir\'ag (in the paper ``Zeros of the i.i.d. Gaussian power series: a conformally invariant determinantal process'' {\it Acta Math.} (2005) 194, 1-35) discovered the striking fact that the zeros of the random power series $f(z) = \sum_{n=1}^\infty \xi_n z^{n-1}$ in the complex unit disc $\mathbb{D}$ constitute a determinantal point process. The study of the zeros of the general random series 
 $f(z)$ where the restriction of independence is relaxed upon the random variables $(\xi_n)$ is an important open problem. This paper proves that if $(\xi_n)$  is an infinite sequence of complex Gaussian random variables such that their covariance matrix is invertible and its inverse is a Toeplitz matrix, then the zero set of $f(z)$ constitutes a determinantal point process with the same distribution as the case of i.i.d variables studied by Peres and Vir\'ag. The arguments are based on some interplays between Hardy spaces and reproducing kernels. Illustrative examples are constructed from classical Toeplitz matrices and the classical fractional Gaussian noise.

 \end{abstract}

\section{Introduction}

Given a sequence of independent and identically distributed standard complex Gaussian random variables $(\xi_n)$, consider the Gaussian power series 
   $f(z) = \sum_{n=1}^\infty \xi_n z^{n-1}$ defined in the open unit disc $\mathbb{D} = \{z \in \mathbb{C}: |z| < 1\}$ and the zero set of $f(z)$, that is, 
$$\mathscr{Z} = \{z \in \mathbb{D}: f(z) = 0\}.$$
The zero set $\mathscr{Z}$ constitutes a point process on $\mathbb{D}$. The joint intensity $p$ of the process $\mathscr{Z}$ is defined as
   $$p(z_1,z_2, \ldots, z_n) = \lim_{\epsilon\to 0} \frac{\mathbb{P}_{\epsilon}(z_1, z_2, \ldots, z_n)}{\pi^n \epsilon^{2n}}$$ where 
$\mathbb{P}_\epsilon(z_1, z_2, \ldots, z_n)$ is the probability that simultaneously for all $1\leq i \leq n$, the function $f(z)$ has a zero in the disc of centre $z_i$ and radius $\epsilon>0$. 
Recently, Peres and Vir\'ag \cite{Peres_Virag} obtained the striking fact that the point process $\mathscr{Z}$ is a determinantal process. In fact, they proved that, for all $z_1, z_2, \ldots, z_n$ in $\mathbb{D}$, 
         $$p(z_1, z_2, \ldots, z_n) = \det\left(\frac{1}{\pi(1 - z_k \overline{z_j})^2}\right)_{k,j=1}^n.$$
That is, 
            $$p(z_1, z_2, \ldots, z_n) = \det\left(K(z_k, z_j)\right)_{1\leq k,j\leq n}$$  where $K(z, w) = \pi^{-1}(1 - z \overline{w})^{-2}$ is the classical Bergman kernel in $\mathbb{D}$. (A thorough discussion on determinantal point processes can be found in the book by Hough et al. \cite{Hough}.)
The study of the zeros of the general random series 
 $f(z) = \sum_{n=1}^\infty \xi_n z^{n-1}$ where the restrictions of independence and identical distribution are relaxed upon the random variables $(\xi_n)$ is an important open problem. This paper considers this question in the following context: {\it Determine sequences of dependent Gaussian random variables $(\xi_n)$ such that the zero set of the random series $f(z) = \sum_{n=1}^\infty \xi_n z^{n-1}$ is a determinantal point process as in the case of i.i.d random variables.} We recall that a matrix $(a_{k,j})$ is called a Toeplitz matrix if $a_{k,j}$ depends only on the difference $k-j$, that is, for all $k,j$ and for any integer $\ell$ such that $a_{k+\ell, j+\ell}$ is defined, $a_{k+\ell, j+\ell} = a_{k,j}.$ 
We consider a complex infinite Toeplitz matrix $G$ that is hermitian and positive definite and we  assume that $G$ admits a classical inverse in the sense that there exists a hermitian positive definite matrix $G^{-1}$ such that $G G^{-1} = G^{-1} G = I$. Since $G$ is positive definite and Toeplitz, then there exits a positive definite function $\gamma$ on the integers such that $G_{k,j} = \gamma(k-j) = {\overline{\gamma(j-k)}}$ for all $k,j$. (We shall assume without loss of generality that $G_{k,k} = \gamma(0) = 1$ for all $k$.)
Then by the classical Bochner theorem, one can associate to $G$ a (unique) probability measure $\mu$ on the unit circle $\mathbb{T}$ such that 
\begin{eqnarray*} 
\gamma(n) = \int_{\mathbb{T}} e^{2\pi i n \theta} d\mu(\theta),\,\, \mbox{ for all } n \in \mathbb{Z}.
\end{eqnarray*}
We shall assume throughout that the probability measure $\mu$ satisfies the following condition: 
{\bf Condition (C):} The measure $\mu$ is absolutely continuous and its density $\varphi$ is strictly positive almost everywhere on $\mathbb{T}$ with respect to the Lebesgue measure on $\mathbb{T}$.  

We now consider a discrete-time complex Gaussian process $(\xi_n)_{n\in \mathbb{N}}$ with zero mean, covariance matrix  $G^{-1}$ and zero pseudo-covariance matrix, that is, for all $n,m\in \mathbb{N}$,
      $$\mathbb{E}(\xi_n) = 0, \,\,\,\mathbb{E}(\xi_n \overline{\xi_m}) = \left(G^{-1}\right)_{n,m} 
\mbox{    and   } \mathbb{E}(\xi_n \xi_m)  = 0.$$
The existence of such process $(\xi_n)$ is classical (see for example Miller \cite{Miller} and references therein.) In the case where $G$ (and hence $G^{-1})$ is a real matrix one can simply take a real Gaussian process $(\zeta_n)$ of covariance matrix $G^{-1}$ and write $\xi_n = (\zeta_n + i \zeta_n')/\sqrt{2}$ where $(\zeta_n')$ is an independent copy of $(\zeta_n)$.
We shall consider the Gaussian analytic function 
         $$f(z) = \sum_{n=1}^\infty \xi_n z^{n-1},\,\, z\in \mathbb{D}.$$
Our main finding is that  the zero set of the Gaussian analytic function $f(z)$
is a determinantal point process governed by the Bergman kernel just as it is for the case of i.i.d random variables.         
The main result of this paper is the following.
 \begin{theorem}\label{mainth} 
 Let $G$ be an infinite, invertible, hermitian and positive definite Toeplitz matrix such that the associated probability measure $\mu$ satisfies condition $(C)$ and the inverse $G^{-1}$ is such that
\begin{eqnarray} \label{sdse34rfs}
\sup_{n,m} |(G^{-1})_{n,m}| < \infty.
\end{eqnarray}
If $(\xi_n)_{n \in \mathbb{Z}}$ is a centred complex Gaussian process with covariance matrix $G^{-1}$ and zero pseudo-covariance matrix, then the  zero set of the Gaussian analytic function
$$f(z) = \sum_{n=1}^\infty \xi_n z^{n-1},\,\, z\in \mathbb{D}$$
  is a determinantal point process governed by the Bergman kernel. That is the joint intensity $p$ of the zeros of $f(z)$ is given by
       $$p(z_1, z_2, \ldots, z_n) =  \det\left(\frac{1}{\pi(1 - z_k \overline{z_j})^2}\right)_{1\leq k,j\leq n},\,\,\, z_1, z_2, \ldots, z_n \in \mathbb{D}.$$ 
\end{theorem}
If $(\zeta_n)$ is a centred  Gaussian process with covariance matrix $G$ (that is a Toeplitz matrix), the zero set of $f(z) =  \sum_{n=1}^\infty \zeta_n z^{n-1}$ does not necessarily have the same distribution as $\sum_{n=1}^\infty \chi_n z^{n-1}$ for i.i.d Gaussian variables $(\chi_n)$. (This is discussed in \cite{Mukeru_al}. An example is also given in section 6.1.) However if we take instead of $(\zeta_n)$  a sequence $(\xi_n)$ with covariance matrix $G^{-1}$ then the corresponding zero set has the same distribution as for i.i.d. Gaussian variables.   
 This implies that in terms of the corresponding zero sets sequences of Gaussian variables whose covariance matrix is the inverse of a Toeplitz matrix are more closer to the sequence of i.i.d Gaussian variables than sequences of variables whose covariance matrix is a Toeplitz matrix. This looks awkward but one should remember that the distribution of a Gaussian vector depends more directly on the inverse of its covariance matrix rather than the covariance matrix itself. 
Lemma \ref{lemmaone} in section \ref{sdsefer} gives an interesting property of sequences of Gaussian variables whose covariance is an inverse of a Toeplitz matrix. It looks like such sequences are of some independent interest that requires further investigation. 

The question of dependent random variables can be raised in connection with other determinantal point processes, for instance the point processes obtained by Krishnapur \cite{Krishnapur}. It is also relevant in the context of 
Pfaffian processes studied by  Matsumoto and Shirai \cite{Matsumoto}. 
 The rest of the paper is organised as follows. Section 2 contains some basic well-known facts about the zeros of the $f(z)$ and the Szeg\"o kernel. Section 3 contains a connection between the covariance kernel of $f(z)$ and a Hardy space defined by the spectral measure of $(\xi_n)$. In section 4 we obtain some  properties of the sequence $(\xi_n)$. In section 5 we provide an important connection between the classical Mobius transformation and the covariance kernel of $f(z)$ which is key to the proof of the main result. The last section contains some examples that illustrate the main result.

\section{Szeg\"o kernel and Hardy spaces}   
                  
The starting point in the study of the zeros of any zero-mean Gaussian analytic function $f$ in a planar domain is the following general expression for its joint intensity function (Peres and Vir\'ag \cite{Peres_Virag}): 
    \begin{eqnarray} \label{eqnads23001}
     p(z_1, z_2, \ldots, z_n) = \frac{\mathbb{E}\left(|f'(z_1) f'(z_2) \ldots f'(z_n)|^2\,|f(z_1) = f(z_2) = \ldots=f(z_n) = 0\right)}{\pi^n \det(A)}
     \end{eqnarray}
 or equivalently  
   \begin{eqnarray} \label{eqnads231}
     p(z_1, z_2, \ldots, z_n) = \frac{\mbox{perm}(C - B A^{-1} B^*)}{\pi^n \det(A)}
     \end{eqnarray}
  where $A, B$ and $C$ are the $n \times n$ matrices
   $$A = (\mathbb{E}(f(z_k) \overline{f(z_j)})), B = (\mathbb{E}(f'(z_k) \overline{f(z_j)})) \mbox{ and } C = (\mathbb{E}(f'(z_k) \overline{f'(z_j)}))$$ 
and $\mbox{perm}$ denotes the permanent of a matrix. For the classical Gaussian power series $f(z) = \sum_{n=1}^\infty \xi_n z^{n-1}$ with independent and identically distributed (i.i.d) random variables $(\xi_n)$,
    $$ \mathbb{E}(f(z) \overline{f(w)})) = \sum_{n=0}^\infty (z \overline{w})^{n-1} = \frac{1}{1 - z \overline{w}}, \,\, z, w \in \mathbb{D}.$$
That is, $$ \mathbb{E}(f(z) \overline{f(w)})) = \mathbb{K}(z, w)$$ where $\mathbb{K}$ is the classical Szeg\"o kernel. This means that the covariance kernel of $f(z)$  is the Szeg\"o kernel. 
The classical Hardy space
 $\mathbb{H}^2(\mathbb{D})$  is  the class of holomorphic functions $f$ in the unit disc $\mathbb{D}$ for which 
           $$\sup_{0 \leq r < 1} \int_{\mathbb{T}} \left|f(r e^{2\pi i \theta})\right|^2 d\theta = \lim_{r \to 1} \int_{\mathbb{T}} \left|f(r e^{2\pi i \theta})\right|^2 d\theta  < \infty$$ where 
      $\mathbb{T}$ is the unit circle $\mathbb{R}/\mathbb{Z}$. Equivalently $\mathbb{H}^2(\mathbb{D})$ is the class of holomorphic functions  $f(z) = \sum_{n=0}^\infty a_n z^n$,  $z \in \mathbb{D}$, $a_n \in \mathbb{C}$ such that $\sum_{n=0}^\infty |a_n|^2 < \infty.$ 
It is a Hilbert space with the inner product: 
 \begin{eqnarray} \label{ew23edfr1e1}
\langle f, g\rangle = \sum_{n=0}^\infty a_n \overline{b_n},\,\,\,\, f = \sum_{n=0}^\infty a_n z^n, \,\,g = \sum_{n=0}^\infty b_n z^n.
\end{eqnarray}
Any function $f(z) = \sum_{n=0}^\infty a_n z^n$ in $\mathbb{H}^2(\mathbb{D})$ is such that its radial limit
$$\tilde{f}(\theta) = \lim_{r\to 1} f(r e^{2\pi i \theta}) = f(e^{2\pi i \theta})=  \sum_{n=0}^\infty a_n e^{2\pi i n \theta}$$
exists almost everywhere in $\mathbb{T}$ and $\tilde{f} \in L^2(\mathbb{T})$. Moreover
$$\langle f, g\rangle = \int_{\mathbb{T}} \tilde{f}(\theta) \overline{\tilde{g}(\theta)} d\theta =  \int_{\mathbb{T}} f(e^{2\pi i \theta}) \overline{g(e^{2\pi i \theta})} d\theta$$ and 
 $$\|f\|^2_{H^2(\mathbb{D})} = \|\tilde{f}\|_{L^2(\mathbb{T})} = \sum_{n=0}^\infty |a_n|^2.$$
(See Katznelson \cite[p 98]{Katznelson}.)    

It is well-known that $\mathbb{H}^2(\mathbb{D})$  is a reproducing kernel Hilbert space whose kernel is the Szeg\"o kernel. This means that for each $y \in \mathbb{D}$ and for each $f \in \mathbb{H}^2(\mathbb{D})$, the function
$\mathbb{K}(.,y): \mathbb{D} \to \mathbb{C}$ defined by $\mathbb{K}(.,y)(x) = \mathbb{K}(x, y)$ is such that
     $$f(y) = \langle f, \mathbb{K}(.,y) \rangle.$$
(We refer to the book by Paulsen \cite{Paulsen} for a background on reproducing kernel Hilbert spaces.)
The main argument of Peres and Vir\'ag is to make use of these connections between the Hardy space $\mathbb{H}^2(\mathbb{D})$ and the Szeg\"o kernel. 

\section{Inverse Toeplitz matrices and weighted Hardy spaces}  
In the general case where the covariance matrix of the variables $(\xi_n)$ is the matrix $G^{-1}$ (where $G$ is an infinite hermitian and positive definite Toeplitz matrix), the covariance kernel of the function   $f(z)$ is given by
  \begin{eqnarray} \label{sdds34rewdss}
\mathbb{K}_G(z,w) = \mathbb{E}\left(f(z) \overline{f(w)}\right) =  \sum_{k,j=1}^\infty \left(G^{-1}\right)_{k,j} z^{k-1} \left(\overline{w}\right)^{j-1}, \,\, z, w \in \mathbb{D}.
\end{eqnarray}
For the convergence of the series in   (\ref{sdds34rewdss}) it is enough to assume that 
 $\sup_{k,j} |(G^{-1})_{k,j}| < \infty$
(that is condition (\ref{sdse34rfs}) in Theorem \ref{mainth}.)
One can write        $$ \mathbb{K}_G(z,w) =  Z^{T} G^{-1} \overline{W}$$
where $Z^T = (1, z, z^2, \ldots)$ and $W = (1, w, w^2, \ldots)^T$. 
Clearly in the particular case where $G$ is the identity matrix $\mathbb{K}_G$ is the Szeg\"o kernel . An important tool in the proof of the main result is the fact that there is a reproducing kernel Hilbert space whose kernel is the covariance kernel $\mathbb{K}_G$. Assume that the Toeplitz matrix $G$ is given by $G_{k,j} = \gamma(k-j) = {\overline{\gamma(j-k)}}$ for  a function $\gamma$ defined on the integers. We shall assume without loss of generality that $G_{k,k} = \gamma(0) = 1$ for all $k$.
Since $G$ is a Toeplitz matrix and it is positive definite then by the classical Bochner theorem, there exists a probability measure $\mu$ on the unit circle $\mathbb{T}$ such that 
\begin{eqnarray} \label{sdwdsdserfds}
\gamma(n) = \int_{\mathbb{T}} e^{2\pi i n \theta} d\mu(\theta),\,\, \mbox{ for all } n \in \mathbb{Z}.
\end{eqnarray}
We have assumed throughout that $\mu$ is absolutely continuous. Its density $\varphi$ is called the spectral density function of the matrix $G$ and is such that 
\begin{eqnarray} \label{dsdsedererss}
\gamma(n) = \int_{\mathbb{T}} e^{2\pi i n \theta} \varphi(\theta)  d\theta,\,\, \mbox{ for all } n \in \mathbb{Z}
\end{eqnarray}
which (under some conditions)  implies in return that 
\begin{eqnarray} \label{dsdswedwew334}
\varphi(\theta) = \sum_{n \in \mathbb{Z}} \gamma(n) e^{-2\pi  i n \theta},\,\, \theta \in \mathbb{T}. 
\end{eqnarray}
Consider the sub-space $H^2_G(\mathbb{D})$ of the Hardy space $H^2(\mathbb{D})$ of functions $f(z) = \sum_{n=0}^\infty a_n z^n$, $z \in \mathbb{D}$, $a_n \in \mathbb{C}$ such that 
 $$\|f(e^{2\pi i \theta})\|^2_{L^2(\mu)} = \int_{\mathbb{T}} |f(e^{2\pi i \theta})|^2 d\mu(\theta)  < \infty$$ and set
  $$\|f\|_{H^2_G(\mathbb{D})} = \|f(e^{2\pi i \theta})\|_{L^2(\mu)}.$$
Clearly if $$\|f\|^2_{H^2_G(\mathbb{D})} = \int_{\mathbb{T}} |f(e^{2\pi i \theta})|^2 \varphi(\theta) d\theta = 0,$$ then the fact that $\varphi > 0$ almost everywhere in $\mathbb{T}$ yields
$$\|f\|^2_{H^2(\mathbb{D})} = \int_{\mathbb{T}} |f(e^{2\pi i \theta})|^2 d\theta = 0.$$ 
This implies that $f = 0$ everywhere in $\mathbb{D}$.  Moreover the norm $\|.\|_{H^2_G(\mathbb{D})}$ is complete. 
 Indeed, let $(f_k)$  be a sequence of functions in $H^2_G(\mathbb{D})$ such that 
$$\lim_{k,j\to \infty} \|f_k - f_j\|^2_{H^2_G(\mathbb{D})} = \lim_{k,j\to \infty} \|f_k(e^{2\pi i \theta}) - f_j(e^{2\pi i \theta})|^2 _{L^2(\mu)}=  0.$$ Since $L^2(\mu)$ is complete, then the sequence $(f_k(e^{2\pi i \theta}))$ has a limit $\ell: \mathbb{T} \to \mathbb{C}$ in $L^2(\mu)$. That is,
 \begin{eqnarray*} 
\lim_{k\to \infty} \int_{\mathbb{T}} |f_k(e^{2\pi i \theta}) - \ell(\theta)|^2 \varphi(\theta) d\theta =  0.
\end{eqnarray*}
Again since $\varphi >0$ almost everywhere, then 
\begin{eqnarray} \label{kaziba1}
\lim_{k\to \infty} \int_{\mathbb{T}} |f_k(e^{2\pi i \theta}) - \ell(\theta)|^2 d\theta =  0.
\end{eqnarray}
 This means that $\ell$ is also the limit of the sequence  $(f_k(e^{2\pi i \theta}))$ in 
$L^2(\mathbb{T})$. It follows that $\ell \in L^2(\mathbb{T})$. We can consider the Fourier series of $\ell$ and write
 $\ell(\theta) = \sum_{n = -\infty}^{\infty} \hat \ell(n) e^{2\pi i n \theta} $ (this series converges in $L^2(\mathbb{T})$) and $$\sum_{n = -\infty}^{\infty} |\hat \ell(n)|^2  = \|\ell\|^2_{L^2(\mathbb{T})} < \infty.$$ 
Set
    $$f_{n,k}(z) = \sum_{n=0}^\infty a_{n,k} z^n.$$ 
  Then (\ref{kaziba1}) implies
               $$\lim_{k \to \infty} \sum_{n \geq 0} |a_{n,k} - \hat \ell(n)|^2 + \sum_{n < 0} |\hat \ell(n)|^2 = 0.$$
Hence $ \hat \ell(n) = 0$ for all $n < 0$. Then consider the function 
 $g(z) = \sum_{n = 0}^\infty \hat \ell(n) z^n$. It is now clear that $(f_k)$ converges to $g$ both in $H^2(\mathbb{D})$ and $H^2_G(\mathbb{D})$. 

We define the inner product on $H^2_G(\mathbb{D})$ by:
 \begin{eqnarray} \label{ew23edfr1eswe}
\langle f, g\rangle =  \int_{\mathbb{T}} f(e^{2\pi i \theta}) \overline{g(e^{2\pi i \theta})} d\mu(\theta).
\end{eqnarray}
Clearly $H^2_G(\mathbb{D})$ is a Hilbert space. We want to show that $\mathbb{H}^2_G(\mathbb{D})$ is in fact a reproducing kernel Hilbert space whose kernel is $\mathbb{K}_G$.  
   \begin{theorem} \label{sdsd34edsds}
The space $\mathbb{H}^2_G(\mathbb{D})$ is a reproducing kernel Hilbert space whose kernel is 
  $\mathbb{K}_G$ given by
  $$ \mathbb{K}_G(z,w) =  Z^{T} G^{-1} \overline{W},\,\,z, w \in \mathbb{D}$$
    with $Z = (z^n)_{n \in \mathbb{N}}$ and $W = (w^n)_{n \in \mathbb{N}}$.
\end{theorem}

\begin{proof}
We shall first prove that the monomials $z^n$ are in the reproducing kernel Hilbert space associate to $\mathbb{K}_G$. That is, 
    $$z^n = \int_{\mathbb{T}} e^{2\pi i n \theta} \,\,\overline{\mathbb{K}_G(e^{2\pi i \theta}, z)} d\mu(\theta).$$
Since for all $w, y \in \mathbb{D}$,
  $$ \mathbb{K}_G(w, y) = \sum_{k, j = 1}^\infty (w)^{k-1} (\overline{y})^{j-1} (G^{-1})_{k,j} = \overline{\mathbb{K}_G(y, w)},$$ then 

\begin{eqnarray*}
\int_{\mathbb{T}} e^{2\pi i n \theta} \,\,\overline{\mathbb{K}_G(e^{2\pi i \theta}, z)} d\mu(\theta) & = & \int_{\mathbb{T}} e^{2\pi i n \theta} \,\,\mathbb{K}_G(z, e^{2\pi i \theta}) d\mu(\theta) \\
 & = &  \sum_{j=1}^{\infty} z^{j-1} \sum_{k=1}^\infty (G^{-1})_{k,j} \int_{\mathbb{T}}  e^{2\pi i (n - k+1) \theta} d\mu(\theta) \\
& = & \sum_{j=1}^{\infty} z^{j-1} \sum_{k=1}^\infty (G^{-1})_{k,j} \,\gamma(n-k+1) \\
& = & \sum_{j=1}^{\infty} z^{j-1} \sum_{k=1}^\infty (G^{-1})_{k,j}\, G_{n+1,k} \\
& = & z^n.\\
\end{eqnarray*}
To complete of the proof it suffices to determine an orthonormal basis $\{P_k(z): k=1,2,\ldots\}$ of $H^2_G(\mathbb{D})$ and prove that it is the case that
          $$\sum_{k=1}^\infty P_k(z) \overline{P_k(w)} = \mathbb{K}_G(z, w),\,\,\mbox{ for all } z, w \in \mathbb{D}.$$
First, it is clear that  in the Hilbert space $H^2_G(\mathbb{D})$, 
for all $k,j \in \mathbb{N}$, 
           \begin{eqnarray*}
         \langle z^k,  z^j \rangle & = & \int_{\mathbb{T}}  e^{2\pi i(k-j) \theta}  d\mu(\theta)  \\
                  & = & \gamma(k-j).
 \end{eqnarray*}
Next we shall take $(P_k)$ to be the orthonormal basis deduced from 
the sequence of polynomials $(1, z, z^2, \ldots, z^k, \ldots)$ by the classical Gram--Schmidt process.  
Consider an infinite lower-triangular matrix $A = (a_{k,j})$ (that is $a_{k,j} = 0$ for $j>k$) such that 
 for each $k$,
\begin{eqnarray} \label{sdewewe3reew}
z^{k-1} = a_{k,1} P_1(z) + a_{k,2} P_2(z) + \ldots+ a_{k,k} P_k(z).
\end{eqnarray}
Then since $\{P_k: k=1,2,\ldots\}$ is orthonormal, then for $j\leq k$, 
\begin{eqnarray*}
\langle z^{k-1},  z^{j-1} \rangle  = a_{k,1} \overline{a_{j,1}} + a_{k,2} \overline{a_{j,2}} + \ldots+ a_{k,j} \overline{a_{j,j}}.
\end{eqnarray*}
Moreover, using 
$ \langle z^{k-1},  z^{j-1} \rangle  = \gamma(k-j)$, it follows that
  $$a_{k,1} \overline{a_{j,1}} + a_{k,2} \overline{a_{j,2}} + \ldots+ a_{k,j} \overline{a_{j,j}} = \gamma(k-j) = G_{k,j}.$$
This yields
       $A A^* = G$ where $A^*$ is the conjugate transpose of $A$.
It follows from (\ref{sdewewe3reew}) that 
   \begin{eqnarray*}
  \begin{pmatrix}
   P_1(z)\\
   P_2(z)\\
   P_3(z)\\
   \vdots\\
   
   P_k(z)\\
   \vdots
  \end{pmatrix}
    = A^{-1} 
   \begin{pmatrix}
   1\\
   z\\
   z^2\\
   \vdots\\
   z^{k-1}  \\
\vdots 
   \end{pmatrix} = A^{-1} Z
  \end{eqnarray*}
 for a lower-triangular matrix $A$ such that $A A^* = G$. It is clear that since $A$ is a lower triangular matrix, then $A^{-1}$ is also a lower triangular matrix and moreover $P_k(z)$ is fully determined by the first $k$ rows of $A$.         
Now clearly,
     \begin{eqnarray*}
      \sum_{k=1}^\infty P_k(z) \overline{P_k(w)} &  = & \lim_{n \to \infty} \sum_{k=1}^n P_k(z) \overline{P_k(w)}  \\
& = &\lim_{n \to \infty} \left(\left(A_n\right)^{-1} Z_n\right)^T \overline{\left(A_n\right)^{-1} W_n} \\
& = & \lim_{n\to \infty} Z_n^T (G_n)^{-1} \overline{W_n} \\
& = & Z^T G^{-1} \overline{W}
     \end{eqnarray*}
where $A_n$ (resp. $G_n$)  is the block of $A$ (resp. $G$) consisting of the first $n$ rows and columns of $A$ (resp. $G$) and
$Z_n = (1, z, z^2, \ldots, z^{n-1})$ and $W_n  = (1, w, w^2,  \ldots, w^{n-1})$.
It follows that 
 $$      \sum_{k=1}^\infty P_k(z) \overline{P_k(w)}  = \mathbb{K}_G(z, w)$$ which concludes the proof. 

\end{proof}

\begin{corollary}
The Gaussian analytic function $f(z) = \sum_{n=1}^\infty \xi_n z^{n-1}$ (where $(\xi_n)$ has covariance matrix $G^{-1}$) has the same distribution with the function
                 $g(z) = \sum_{n=1}^\infty \chi_k P_n(z)$ where $(\chi_n)$ is a sequence of standard i.i.d complex Gaussian random variables and $(P_n(z))$ are the polynomials defined by the matrix $G$ as in the proof of Theorem \ref{sdsd34edsds}.              
\end{corollary}
It is so because the two random functions have the same covariance kernel:
 $$\mathbb{E}(f(z) \overline{f(w)}) = \mathbb{E}(g(z) \overline{g(w)}) = \sum_{k=1}^\infty P_n(z) \overline{P_n(w)} = Z^T G^{-1} \overline{W}.$$

 \section{Some properties of the sequence $(\xi_n)$ of covariance $G^{-1}$} \label{sdsefer}
Here we obtain important properties of the sequence $(\xi_n)$ of covariance matrix $G^{-1}$ that will be useful for the proof of main result. 

\begin{lemma}\label{lemmaone}
Assume that  $(\xi_n)_{n\in \mathbb{N}}$ is a centred complex Gaussian process with zero \\ pseudo-covariance and  covariance matrix $G^{-1}$ where $G$ is an infinite hermitian positive definite Toeplitz matrix. Then for each $n \geq 2$, the conditional joint distribution of the sequence $(\xi_n, \xi_{n+1}, \xi_{n+2}, \ldots)$ under the condition $\xi_1 = \xi_2 = \ldots=\xi_{n-1} = 0$ is equal to the unconditional joint distribution of $(\xi_{1}, \xi_{2}, \xi_{3}, \ldots)$. 
That is,
 $$((\xi_n, \xi_{n+1}, \xi_{n+3}, \ldots)|\xi_1 = \xi_2 = \ldots=\xi_{n-1} = 0) \stackrel{d}{=} (\xi_{1}, \xi_{2}, \xi_3, \ldots)$$
 \end{lemma}

\begin{proof}.
Set $$S_1 = (\xi_1, \xi_2, \ldots), \,\, S_n = (\xi_{n}, \xi_{n+1}, \ldots),\,\,\mbox{ for all } n\geq 1.$$
Then it is well-known that the covariance matrix of 
$$(S_n|\xi_1 = \xi_2 = \ldots=\xi_{n-1} = 0)$$ is the Schur complement of $\mbox{Cov}(\xi_1, \xi_2, \ldots, \xi_{n-1})$ in the overall covariance matrix $\mbox{Cov}(S_1) = G^{-1}.$ It is obtained by taking the matrix $G^{-1}$, take its inverse, that is, $G$, delete the rows and columns corresponding to the random variables $\xi_1, \xi_2, \ldots, \xi_{n-1}$ and take the inverse of the resulting matrix. Now deleting the first $n-1$ rows and columns of the infinite Toeplitz matrix $G$ yields the very same matrix $G$. It follows that 
the covariance matrix of $(S_n|\xi_1 = \xi_2 = \ldots=\xi_{n-1} = 0)$ is just $G^{-1}$. This implies that 
 $(S_n|\xi_1 = \xi_2 = \ldots=\xi_{n-1} = 0)$ has the same distribution as $(\xi_1, \xi_2, \ldots)$. 
 \hfill \qed
\end{proof}
An immediate consequence of this lemma is: 
\begin{corollary} \label{corol1}
For any sequence $(\alpha_k)$ of complex numbers and for any integer $n \geq 1$,
 $$\left.\left(\sum_{k=1}^\infty \alpha_{k} \xi_{n+k-1}\right)\right|\xi_1 = \xi_2 = \ldots=\xi_{n-1} = 0)$$ has the same distribution with $$\sum_{k=1}^\infty \alpha_k \xi_{k}$$ provided the involved series converge almost surely. 
In particular,
           $$(\xi_n|\xi_1 = \xi_2 = \ldots=\xi_{k} = 0) \stackrel{d}{=} \xi_{n-k} \,\, \mbox{ for all } 1 \leq k < n.$$
           \end{corollary}

\section{Mobius transformation and the kernel $\mathbb{K}_G$}
\subsection{Mobius transformation}
For $w \in \mathbb{D}$, consider as in Peres and Vir\'ag \cite{Peres_Virag}, the Mobius transformation of the unit disc $\mathbb{D}$:
           $$T_w(z) = \frac{z-w}{1-z \overline{w}},\,\, z\in \mathbb{D}.$$
In the case where $G$ is the identity matrix (or equivalently the random variables $(\xi_k)$ are independent and identically distributed), Peres and Vir\'ag \cite{Peres_Virag} proved the following lemma:
\begin{lemma}[Peres and Vir\'ag] \label{lemmaPV}
Assume that $(\xi_k)$ is the sequence of standard i.i.d complex Gaussian random variables. Then for any $w$ fixed in $\mathbb{D}$, under the condition $f(w) = 0$, the random function $f(z)$ has the same distribution with 
          $$T_{w}(z) f(z) = \left(\frac{z-w}{1- z \overline{w}}\right) f(z),\,\,z \in \mathbb{D},$$ that is, 
$$(f(z)| f(w) = 0) \stackrel{d}{=} T_{w}(z) f(z)$$ where
   $\stackrel{d}{=}$ denotes equality in distribution.  
 In general for $w_1, w_2, \ldots, w_n$ fixed in $\mathbb{D}$, 
        \begin{eqnarray} 
(f(z)| f(w_1)=0, f(w_2)=0, \ldots, f(w_n) = 0) \stackrel{d}{=} T_{w_1}(z) T_{w_2}(z) \ldots T_{w_n}(z) f(z).
\end{eqnarray}   
\end{lemma}
A closer look at Peres and Vir\'ag's proofs reveals that their main result (that is Theorem \ref{mainth} in the  case where $G$ is the identity matrix) is a consequence of Lemma \ref{lemmaPV}. This implies that if we prove that Lemma \ref{lemmaPV} holds true in the  general case of a  Toeplitz matrix $G$, then the same argument as in Peres and Vir\'ag \cite{Peres_Virag} will complete the proof of Theorem \ref{mainth}. 
In other words, in order to prove our main result, it is sufficient to prove that the following lemma holds.  

\begin{lemma} \label{lemmaSM}
Let $G$ be an invertible infinite hermitian Toeplitz matrix such that its associated measure $\mu$ is absolutely continuous with density $\varphi >0$ almost everywhere on $\mathbb{T}$. Assume that $(\xi_k)_{k\in \mathbb{N}}$ is a centred Gaussian process  with covariance matrix $G^{-1}$ and zero pseudo-covariance. Let
      $$f(z) = \sum_{k=1}^\infty \xi_k z^{k-1},\,\,\, z\in \mathbb{D}.$$
Then for any $w$ fixed in $\mathbb{D}$, 
$$(f(z)| f(w) = 0) \stackrel{d}{=} T_{w}(z) f(z).$$ Moreover, 
 for $w_1, w_2, \ldots, w_n$ fixed in $\mathbb{D}$, 
        \begin{eqnarray} 
(f(z)| f(w_1)=0, f(w_2)=0, \ldots, f(w_n) = 0) \stackrel{d}{=} T_{w_1}(z) T_{w_2}(z) \ldots T_{w_n}(z) f(z). 
\end{eqnarray}   
\end{lemma}
It is now an easy matter to prove that Lemma \ref{lemmaSM} yields Theorem \ref{mainth} based on Peres and Vir\'ag arguments. 
\subsection{Proof of Theorem \ref{mainth}.}
For all fixed $z_1, z_2, \ldots, z_n$, $w_1, w_2, \ldots, w_n$ in $\mathbb{D}$,
   the conditional joint distribution of $$(f(z_1), f(z_2), \ldots, f(z_n)|f(w_1) = f(w_2) = \ldots = f(w_n)=0)$$ is equal to the non-conditional joint distribution of $$\left(T_{w_1}(z_1) f(z_1), T_{w_2}(z_2) f(z_2), \ldots, T_{w_n}(z_n) f(z_n)\right).$$ Taking the derivatives, it follows as in \cite[corollary 13]{Peres_Virag} 
that the conditional joint distribution of 
   $$(f{'}(z_1), f{'}(z_2), \ldots, f'(z_n)|f(z_1) = f(z_2) = \ldots = f(z_n) = 0)$$is the same as the unconditional joint distribution of 
   $$(\Upsilon'(z_1) f(z_1), \Upsilon'(z_2) f(z_2), \ldots, \Upsilon'(z_n) f(z_n))$$
where 
           $$\Upsilon(z) = T_{z_1}(z) T_{z_2}(z) \ldots T_{z_n}(z).$$    
This follows from the fact that
         $$T_{z}'(z) = \frac{1}{1- |z|^2}\,\,\mbox{ and } \,\,T_z(z) = 0,\,\,\,\,z\in \mathbb{D}. $$
At this stage, we make use of relation (\ref{eqnads23001}) to obtain
 \begin{eqnarray*} 
     p_0(z_1, z_2, \ldots, z_n) & = & \frac{\mathbb{E}\left(|f'(z_1) f'(z_2) \ldots f'(z_n)|^2\,|f(z_1) = f(z_2) = \ldots=f(z_n) = 0\right)}{\pi^n \det(A)} \\
     & = & \frac{\mathbb{E}\left(|\Upsilon'(z_1) f(z_1) \Upsilon'(z_2) f(z_2) \ldots \Upsilon'(z_n) f(z_n)|^2\right)}{\pi^n \det(A)} \\
     & = & \frac{\mathbb{E}\left(|f(z_1)  f(z_2) \ldots  f(z_n)|^2\right) \prod_{k=1}^n |\Upsilon'(z_k)|^2 }{\pi^n \det(A)} \\
         \end{eqnarray*}
Using the classical Cauchy determinant formula, Peres and Vir\'ag \cite{Peres_Virag}  showed that 
     $$ \prod_{k=1}^n |\Upsilon'(z_k)|  = \det(A_0)$$     
where  
                $$A_0 = \left(\frac{1}{1- z_k \overline{z_j}}\right)_{k,j=1}^n.$$
Moreover since it is well-known that if $X_1, X_2, \ldots, X_n$ are random variables with joint Gaussian distribution with mean 0 and covariance matrix $\Sigma$, 
    $$\mathbb{E}\left(|X_1 X_2 \ldots X_n|^2\right)  = \mbox{perm}(\Sigma),$$ it follows that, 
 \begin{eqnarray} \label{dsdefesdws}
    p_0(z_1, z_2, \ldots, z_n)  = \frac{\mbox{perm}(A)(\det(A_0))^2 }{\pi^n \det(A)}.
     \end{eqnarray}
Now elementary operations on the matrix $A$ yields
      \begin{eqnarray*}
\mbox{perm}(A) & = &  \mbox{perm}(A_0) \prod_{k=1}^n \left(\frac{1}{|1-z_k|^2}\right)\\
\det(A) & = &  \det(A_0) \prod_{k=1}^n \left(\frac{1}{|1-z_k|^2}\right).
\end{eqnarray*}
Hence (\ref{dsdefesdws}) yields
            \begin{eqnarray*} 
    p_0(z_1, z_2, \ldots, z_n)  = \frac{\mbox{perm}(A_0)\det(A_0) }{\pi^n}
     \end{eqnarray*}
and it is proven in Peres and Vir\'ag \cite[rel. (27)]{Peres_Virag} that 
  $$\mbox{perm}(A_0)\det(A_0)  = \det\left(\frac{1}{(1 - z_k \overline{z_j})^2}\right)_{k,j=1}^n.$$
This concludes the proof. \hfill \qed

\subsection{Proof of Lemma \ref{lemmaSM}.} Peres and Virag's proof is based on the invariance property of the Szeg\"o kernel with respect to Mobius transformations that are conformal mappings. This property does not hold for the general kernel $\mathbb{K}_G$. Our proof is more general. 
 (a) In the case where $w = 0$, it is an immediate consequence of Corollary \ref{corol1}. Indeed, 
\begin{eqnarray*}
(f(z)| f(w) = 0)  & = & (f(z)|f(0) = 0) = (f(z)|\xi_1=0) = \left(\sum_{k=2}^\infty \xi_k z^{k-1}|\xi_1 = 0\right)\\
& = & z \left(\sum_{k=2}^\infty \xi_k z^{k-2}|\xi_1 = 0\right)  =  z \left(\sum_{k=1}^\infty \xi_{k+1} z^{k-1}|\xi_1 = 0\right)\\
&\stackrel{d}{=} & z \sum_{k=1}^\infty \xi_{k} z^{k-1} 
\end{eqnarray*}
where the equality in distribution follows from Corollary \ref{corol1}.\\
(b)  For a general $w \in \mathbb{D}$, set 
            $$F(z) = (f(z)|f(w)  = 0), \,z \in \mathbb{D}.$$ 
Clearly, the covariance kernel of the random function $F(z)$ is given by
   \begin{eqnarray*}
\mathbb{E}(F(z) \overline{F(y)}) & = & \mathbb{E}\left(f(z) \overline{f(y)}\right) - \frac{\mathbb{E}\left(f(z) \overline{f(w)}\right) \mathbb{E}\left(f(w) \overline{f(y)}\right)}{\mathbb{E}\left(f(w) \overline{f(w)}\right)}\\
&  = & \mathbb{K}_G(z, y)-\frac{\mathbb{K}_G(z, w)\mathbb{K}_G(w, y)}{\mathbb{K}_G(w, w)}
\end{eqnarray*}
where $z, y \in \mathbb{D}$. This is clearly a kernel function and we shall denote it by $\mathscr{K}_1$, that is:
  $$\mathscr{K}_1(z, y) = \mathbb{K}_G(z, y)-\frac{\mathbb{K}_G(z, w)\mathbb{K}_G(w, y)}{\mathbb{K}_G(w, w)},\,\,z, y\in \mathbb{D}.$$
We  also denote by $\mathscr{K}_2$ the covariance kernel of the random function $T_w(z) f(z)$, that is,
\begin{eqnarray*}
\mathscr{K}_2(z, y) &  = & \mathbb{E}\left(T_w(z) f(z) \overline{T_w(y) f(y)}\right)\\
                    & =  & T_w(z) \mathbb{K}_G(z, y)  \overline{T_w(y)}.
\end{eqnarray*}
To show that the (Gaussian) random functions $F(z)$ and  $T_w(z) f(z)$ have the same distribution, it is sufficient to show that their  covariance kernels $\mathscr{K}_1$ and $\mathscr{K}_2$ are identical. So we shall prove the following important identity: For all $w, z, y \in \mathbb{D}$,
     $$\mathbb{K}_G(z, y)-\frac{\mathbb{K}_G(z, w)\mathbb{K}_G(w, y)}{\mathbb{K}_G(w, w)}  = T_w(z) \mathbb{K}_G(z, y)  \overline{T_w(y)}.$$      
We shall prove that  $\mathscr{K}_1$ and $\mathscr{K}_2$ are both reproducing kernels of a certain Hilbert space, namely the subspace $\mathscr{H}$ of $H^2_G(\mathbb{D})$ of the functions that vanish at the point $w$. That is, 
          $$\mathscr{H} = \{g \in H^2_G(\mathbb{D}): g(w) = 0\}.$$  
(To see that  $\mathscr{H}$ is closed in  $H^2_G(\mathbb{D})$, assume that $g_k = \sum_{n=0}^\infty a_{n,k} z^n$ converges to $g = \sum_{n=0}^\infty b_n z^n$. Then as discussed earlier, this convergence also holds with respect to the $H^2(\mathbb{D})$-norm and hence $\lim_{k \to \infty} \sum_{n=0}^\infty |a_{n,k} - b_n|^2 = 0$. Then 
for each $z \in \mathbb{D}$ (by the Cauchy-Schwarz inequality)
 $$\lim_{k \to \infty} |g_{k}(z) - g(z)|^2 = \lim_{k \to \infty} |\sum_{n=0}^\infty (a_{n,k} - b_n) z^k|^2 \leq \lim_{k \to \infty} \sum_{n=0}^\infty |a_{n,k} - b_n|^2 \sum_{n=0}^\infty |z|^{2n} = 0.$$ In particular since $g_{k}(w) = 0$ for all $n$, then $g(w) = 0$ and hence $g \in \mathscr{H}$.)
 
To prove the claim that $\mathscr{K}_1$ and $\mathscr{K}_2$ are reproducing kernels $\mathscr{H}$, we shall prove that for any function $g \in \mathscr{H}$ and any $y \in \mathbb{D}$,
\begin{eqnarray} \label{sdsde34rfds}
g(y) = \langle g,  \mathscr{K}_1(.,y)\rangle  = \langle g,  \mathscr{K}_2(.,y)\rangle
\end{eqnarray}
 where
$\mathscr{K}_1(.,y)$ is  the function defined by
   $$\mathscr{K}_1(.,y): \mathbb{D} \to \mathbb{C},\, z \mapsto \mathscr{K}_1(z,y)$$ and similarly for $\mathscr{K}_2.$
Clearly,
         \begin{eqnarray*}
\langle g,  \mathscr{K}_1(.,y)\rangle  & = & \left\langle g,  \mathbb{K}_G(., y)-\frac{\mathbb{K}_G(., w)\mathbb{K}_G(w, y)}{\mathbb{K}_G(w, w)}\right \rangle  \\
& = &  \left\langle g,  \mathbb{K}_G(., y)\right\rangle -  \left(\frac{\overline{\mathbb{K}_G(w, y)}}{\overline{\mathbb{K}_G(w, w)}}\right)  \left\langle g, \mathbb{K}_G(., w)\right \rangle\\
& = & g(y) - \left(\frac{\overline{\mathbb{K}_G(w, y)}}{\mathbb{K}_G(w, w)}\right)\, g(w)\\
& = & g(y)  
  \end{eqnarray*}       
since $g(w) = 0$.\\
For the kernel $\mathscr{K}_2$, we shall make use of the explicit  inner product in $H^2_G(\mathbb{D})$ in relation (\ref{ew23edfr1eswe}), and show that
\begin{eqnarray}\label{sdsdesafa1}
g(y) = \int_{\mathbb{T}} g(e^{2 \pi i\theta}) \overline{\mathscr{K}_2(e^{2 \pi i \theta},y)} d\mu(\theta).
\end{eqnarray}
Clearly, by definition of the kernel $\mathscr{K}_2$, 
\begin{eqnarray} \label{ewdsdaswe45}
\int_{\mathbb{T}} g(e^{2 \pi i\theta}) \overline{\mathscr{K}_2(e^{2 \pi i \theta},y)} d\mu(\theta)
= T_w(y) \int_{\mathbb{T}} g(e^{2 \pi i\theta}) \overline{T_w(e^{2 \pi i \theta})} \, \overline{\mathbb{K}_G(e^{2 \pi i \theta},y)} d\mu(\theta).
\end{eqnarray}
Now note that 
  $$\overline{T_w(e^{2 \pi i \theta})} =  \frac{e^{-2 \pi i \theta} -\overline{w}}{1 -  w  e^{-2 \pi i \theta}} = \frac{1 -\overline{w} e^{2 \pi i \theta} }{e^{2 \pi i \theta} -  w}.$$
 Therefore
     \begin{eqnarray*}
  \langle g,  \mathscr{K}_2(.,y)\rangle & = &  T_w(y)  \int_{\mathbb{T}} g(e^{2 \pi i\theta}) \left(\frac{1 -\overline{w}\  e^{2 \pi i \theta} }{e^{2 \pi i \theta} -  w}\right) \, \overline{\mathbb{K}_G(e^{2 \pi i \theta},y)} d\mu(\theta)\\
      & = &  T_w(y) \left \langle g\ p, \mathbb{K}_G(., y)\right\rangle
     \end{eqnarray*} 
where $p$ is the function defined by 
 $$p(z) = \frac{1 -\overline{w}  z }{z -  w} = \frac{1}{T_w(z)},\,\,\,z \in \mathbb{D}.$$ 
 
Since $g \in H^2_G(\mathbb{D})$, then the product $g p$ is also in $H^2_G(\mathbb{D})$. Indeed,
  $$\int_{\mathbb{T}} |g(e^{2\pi i \theta})|^2 |p(e^{2\pi i \theta})|^2 d\mu(\theta) \leq C \int_{\mathbb{T}} |g(e^{2\pi i \theta})|^2  d\mu(\theta) < \infty$$ where
  $$C = \sup_{\theta \in \mathbb{T}} \left|p(e^{2 \pi i \theta})\right|^2 =  \sup_{\theta \in \mathbb{T}} \left|\frac{1 -\overline{w}  e^{2\pi i \theta }}{e^{2\pi i \theta} -  w}\right|^2 < \infty$$
(because $w \in \mathbb{D}$ and hence $e^{2\pi i \theta} -  w \ne 0$). 
It follows that  
 $$\left \langle g p, \mathbb{K}_G(., y)\right\rangle = g(y) p(y) = g(y) \left(T_w(y)\right)^{-1}.$$
 Hence
       $$ \langle g,  \mathscr{K}_2(.,y)\rangle =  T_w(y) g(y) \left(T_w(y)\right)^{-1} = g(y).$$ 
This yields (\ref{sdsde34rfds}) and concludes the proof of 
           $$F(z) = (f(z)| f(w) = 0) \stackrel{d}{=} T_w(z) f(z).$$ 
Now the general case that
$$(f(z)| f(w_1)=0, f(w_2)=0, \ldots, f(w_n) = 0) \stackrel{d}{=} T_{w_1}(z) T_{w_2}(z) \ldots T_{w_n}(z) f(z)$$
             follows immediately by an induction argument.   This concludes the proof of  Lemma \ref{lemmaSM}. \hfill \qed

\section{Illustrating examples}

\subsection{Explicit inverse of tridiagonal Toeplitz matrices}\label{exampl1}
Given a real number $q$ such that $|q|< 1/2$, consider the Toeplitz matrix
  $$G = \left(\gamma(k-j)\right)_{k,j=1}^\infty$$ where 
  \begin{eqnarray*}
         \gamma(k) = \left \{ \begin{array}{cc}
                                               1 &  \mbox{ if } k=0 \\
                                                q &  \mbox{ if } |k| = 1\\
                                                0 & \mbox{ otherwise. }
                                               \end{array}
                     \right.     
 \end{eqnarray*}
The spectral density function $\varphi$ of $G$ (i.e. the density of the corresponding measure $\mu$) is given by
$$\varphi(\theta) = 1 + q e^{2 \pi i \theta} + q e^{-2 \pi i \theta} = 1 + 2 q \cos(2\pi \theta),\,\, \theta \in \mathbb{T}.$$
Let $G_n$ be the submatrix of $G$ formed by its first $n$ rows and first $n$ columns. Then the inverse $G_n^{-1}$ of $G_n$ is the symmetric matrix given by (see  da Fonseca and Petronilho \cite{da Fonseca}): 
  $$\left(G_n^{-1}\right)_{k, j} = (-1)^{k+j} \frac{q^{j-k}}{|q|^{j-k+1}} \frac{U_{k-1}(\alpha) U_{n-j}(\alpha)}{U_n(\alpha)},\,\,\,1\leq k \leq j\leq n$$
where $$\alpha = \frac{1}{2 |q|}$$ and $(U_k)$ is the sequence of Chebyshev polynomials of second kind given by:
     \begin{eqnarray*}
      U_0 & = &  1\\
      U_1(x) & = & 2 x\\
      U_{k+1} (x) & = & 2 x U_k(x) - U_{k-1}(x),\,\,k = 1,2,\ldots
          \end{eqnarray*}
Explicitly, for $|x| > 1$, 
       $$U_k(x) = \frac{\left(x + \sqrt{x^2-1}\right)^{k+1} - \left(x - \sqrt{x^2-1}\right)^{k+1}}{2 \sqrt{x^2-1}}.$$
Taking the limit of $G_n^{-1}$ as $n \to \infty$, it can be easily checked that the infinite matrix $G$ is indeed invertible and its inverse is the infinite symmetric matrix given for $k \leq j$ by
  \begin{eqnarray*}
\left(G^{-1}\right)_{k,j} & = &  \lim_{n\to \infty} \left(G_n^{-1}\right)_{k,j}\\
&  = & \frac{(-2 q)^{j - k} \left(1 + \sqrt{1 - 4 |q|^2}\right)^{-j} \left(\left(1 + \sqrt{1 - 4 |q|^2}\right)^k - \left(1 - \sqrt{1 - 4 |q|^2}\right)^k\right)}{\sqrt{1 - 4|q|^2}}. 
   \end{eqnarray*}
For example, if $q = -\frac{1}{3}$, the inverse of the infinite matrix $G$  is given by 
 $$\left(G^{-1}\right)_{k, j} = \left(\frac{3}{\sqrt{5}}\right) \left(\frac{3 - \sqrt{5}}{2}\right)^j \left(\left(\frac{3 + \sqrt{5}}{2}\right)^
   k-\left(\frac{3 - \sqrt{5}}{2}\right)^{k}\right),\,\,\mbox{ for } j \geq k.$$
Now some involved (but elementary) calculations yield that the kernel function $\mathbb{K}_G$ defined by the matrix $G$ is given by
       \begin{eqnarray}\label{sdwwe34r}
\mathbb{K}_G(z, w) = Z^T G^{-1} \overline{W} = \frac{\psi(z) \overline{\psi(w)}}{1 - z \overline{w}}
\end{eqnarray}
 where  $\psi$ is the function defined in the unit disc $\mathbb{D}$ 
by
  $$\psi(z) =\left(\frac{2}{|q|}\right)^{1/2} \left(\frac{1}{a + b z}\right)$$ with
 \begin{eqnarray*}
 a & = & \sqrt{|q|^{-1} + \sqrt{q^{-2} - 4}} \\
 b & = & (2/a)\sign(q).
 \end{eqnarray*}
For example for $q = -1/3$,       
$$\psi(z) = \frac{5^{1/4}(3 + \sqrt{5}) \left(\frac{3}{2(5 + 3\sqrt{5})}\right)^{1/2}}{z - \frac{3+ \sqrt{5}}{2}}.$$
Note that relation (\ref{sdwwe34r}) means that
        $$\mathbb{K}_G(z, w) = \psi(z) \mathbb{K}(z, w) \overline{\psi(w)}$$ where $\mathbb{K}$ is the classical Szeg\"o kernel (which is the covariance kernel associated to $\sum_{n=1}^\infty \zeta_n z^{n-1}$ for i.i.d $(\zeta_n)$).      
Then for a sequence $(\xi_n)$ of Gaussian random variables with covariance matrix $G^{-1}$, this implies that  the random functions $\sum_{n=1}^\infty \xi_n z^{n-1}$  and  $\psi(z) \sum_{n=1}^\infty \zeta_n z^{n-1}$ have the same distribution, that is,
 $$\sum_{n=1}^\infty \xi_n z^{n-1}  \stackrel{d}{=}  \psi(z) \sum_{n=1}^\infty \zeta_k z^{k-1}$$
 (for i.i.d $(\zeta_n)$). Since clearly $\psi(z) \ne 0$ everywhere in $\mathbb{D}$, it follows that the zeros of $\sum_{n=1}^\infty \xi_n z^{n-1}$ have the same distribution as the zeros of $\sum_{n=1}^\infty \zeta_n z^{n-1}$ and therefore they constitute a determinantal point process as predicted by Theorem 1. 

In particular the intensity of the zeros of $f(z)$ is also given 
$$p(z) = \frac{1}{\pi(1 - |z|^2)^2},\,\,z \in \mathbb{D}$$ as it is the case for i.i.d variables.
To emphasize that we need to consider the $G^{-1}$ as the covariance matrix instead of $G$, consider a sequence of Gaussian variables $(\tau_n)$ with covariance matrix $G$ and the function 
$g(z) = \sum_{n=1}^\infty \tau_n z^n$. Using relation (\ref{eqnads231}), it is easy to derive that the intensity of the corresponding zero set is
  $$p(z) = \frac{1}{\pi(1 - |z|^2)^2} \left(1 - \frac{q^2(1 - |z|^2)^2}{(1 + q z + q \overline{z})^2}\right),\,\,\,\, z \in \mathbb{D}.$$
Hence clearly the zeros of  $g(z)$ do not have the same distribution as the zeros of  $f(z) = \sum_{n=1}^\infty \zeta_n z^{n-1}$ for i.i.d $(\zeta_n)$ since the corresponding intensity is $p(z) = \pi^{-1}(1 - |z|^2)^{-2}$.  

Finally since the function $\varphi(\theta) = 1 + 2 q \cos(2 \pi \theta)$ is bounded on $\mathbb{T}$, then the set 
$H^2_G(\mathbb{D})$ is equal to $H^2(\mathbb{D})$ but with a different norm: 
    $$\|g\|^2_{H^2_G(\mathbb{D})} = \int_{\mathbb{T}} |g(\theta)|^2 \varphi(\theta) d\theta = \int_{\mathbb{T}} |g(\theta)|^2 (1 + 2 q \cos(2 \pi \theta))  d\theta.$$
Therefore with this norm, $H^2(\mathbb{D})$ is the reproducing kernel Hilbert space given by the kernel
          $$K_G(z, w) = \frac{\psi(z) \overline{\psi(w)}}{1 - z \overline{w}} $$ where the function $\psi(z)$ is given above.

\subsection{Inverse of the Kac-Murdock-Szeg\"o matrix}
The same property is also observed for  the classical Kac-Murdock-Szeg\"o matrix.  It is the Toeplitz matrix $G$ defined for a complex number $q$ by 
        $$ (G)_{k,j} =   q^{|k-j|}.$$ 
For $q$ real with $|q| < 1$, the spectral density function $\varphi$ of $G$ is given by
       $$\varphi(\theta) = \sum_{n=-\infty}^\infty q^{|n|} e^{2\pi i n \theta} = \frac{1 - q^2}{1 - 2q \cos(2 \pi \theta) + q^2}. $$
(Here $\varphi(\theta)$ can be seen as the classical Poisson kernel in the unit disc.)
 Under these conditions ($q$ real and $|q|< 1$), $G$ is symmetrical and invertible and its inverse is well-known and given by:
    \begin{eqnarray*}
\left(G^{-1}\right)_{k,j}  = \left\{ \begin{array} {ccc}
            (1-q^2)^{-1} & \mbox{ if } k=j=1\\
 (1+q^2)(1-q^2)^{-1} & \mbox{ if } k=j\geq 2\\
 -q (1-q^2)^{-1} & \mbox{ if } |k-j| = 1\\
 0 & \mbox{ otherwise.} 
 \end{array}
 \right.
 \end{eqnarray*}
(More details on the spectral properties of Kac-Murdock-Szeg\"o matrix for a complex parameter are given in Fikioris \cite{Fikioris}.)
 Then clearly the corresponding kernel is 
$$\mathbb{K}_G(z, w) = \frac{(1- q z)(1 - q \overline{w})}{(1-q^2)(1 - z \overline{w})} = \frac{\psi(z) \overline{\psi(w)}}{1 - z \overline{w}} $$ 
  where $$\psi(z) = \frac{1- q z}{\sqrt{1-q^2}}.$$  
In general for $q \in \mathbb{C}$ with $|q| < 1$, let
    \begin{eqnarray*}
G_{k,j}  = \left\{ \begin{array} {ccc}
            q^{|k-j|} \mbox{ if } k \geq j \\
            (\overline{q})^{|k-j|} \mbox{ otherwise. }
 \end{array}
 \right.
 \end{eqnarray*}  
Then $G$ is hermitian and invertible and its inverse is given by
      \begin{eqnarray*}
\left(G^{-1}\right)_{k,j}  = \left\{ \begin{array} {ccc}
            (1-|q|^2)^{-1} & \mbox{ if } k=j=1\\
 (1+|q|^2)(1-|q|^2)^{-1} & \mbox{ if } k=j\geq 2\\
 -q (1-|q|^2)^{-1} & \mbox{ if } k-j = 1\\
 -\overline{q}(1-|q|^2)^{-1} & \mbox{ if } k - j = -1\\
 0 & \mbox{ otherwise.} 
 \end{array}
 \right.
 \end{eqnarray*}
This yields 
 $$\mathbb{K}_G(z, w) = Z G^{-1} \overline{W} = \frac{(1- q z)(1 - \overline{q} \overline{w})}{(1-|q|^2)(1 - z \overline{w})} = \frac{\psi(z) \overline{\psi(w)}}{1 - z \overline{w}} $$ 
  where $$\psi(z) = \frac{1- q z}{\sqrt{1-|q|^2}}.$$ 
\subsection{Inverse fractional Gaussian noise} 
Given $0 < h < 1$, the classical complex fractional Gaussian noise of Hurst index $h$  is a sequence $\{\Delta_n\}_{n=1}^\infty$ of centred Gaussian random variables such that $\mathbb{E}(\Delta_n \Delta_m) = 0$ and with covariance structure
                \begin{eqnarray} \label{ew32wswaws2121}
\gamma(k) := \mathbb{E}(\Delta_n \overline{\Delta_{n+k}}) = {\scriptstyle\frac{1}{2}}|k+1|^{2h} + {\scriptstyle\frac{1}{2}}|k-1|^{2h} -|k|^{2h},\,\,\,\, k, n\in \mathbb{N}.
\end{eqnarray}
The covariance matrix of $\{\Delta_n\}_{n=0}^\infty$ is the Toeplitz matrix $G$ given by
     $$G = \left(\gamma(k-j)\right)_{k,j=1}^\infty.$$ 
(The particular case $h = 1/2$ corresponds i.i.d random variables.)
It is well-known that the matrix $G$ is invertible (see for example \cite{Dambrogi-ola}.) Unfortunately an explicit inverse of $G$ is not known. Consider its inverse matrix $G^{-1}$. A sequence of Gaussian random variables with covariance matrix $G^{-1}$ shall be called the inverse fractional Gaussian noise of index $h$.  
Some properties of the zeros of the random polynomial $\sum_{k=0}^n \Delta_k x^k$ and the power series $\sum_{k=0}^{\infty} \Delta_n x^n$ where $(\Delta_n)$ is the fractional Gaussian noise are given in \cite{Mukeru_2018} and \cite{Mukeru_al}. Here we are interested in the function $f(z) = \sum_{n=1}^\infty \xi_n z^{n-1}$ where $(\xi_n)$ is the inverse fractional Gaussian noise. 
It is well-known (using an argument by Sinai \cite[Theorem 2.1]{Sinai}) that the matrix $G$ admits a spectral density function $\varphi_h$ given by
   \begin{eqnarray} \label{dsdfesw23}
\varphi_h(\theta) = C(h)|e^{2 \pi i \theta}-1|^2\left(\sum_{n=-\infty}^\infty \frac{1}{|\theta +n|^{2h+1}}\right),\, \theta \in \mathbb{T}, \theta \ne 0
\end{eqnarray}
       where $C(h)$ is a normalising constant given by
        $$C(h) = -\frac{\zeta(-2h)}{2 \zeta(1+2h)}$$ where $\zeta(.)$ is the Riemann zeta function. 
Clearly, 
\begin{eqnarray} \label{dsdfesw2312s}
\varphi_h(\theta) &=& 4 C(h) \left(\sin^2 \pi \theta\right) \sum_{n=0}^\infty \left(\frac{1}{(n+\theta)^{2h+1}} + \frac{1}{(n+ 1-\theta)^{2h+1}}\right) \nonumber\\
           & = &4 C(h) \left(\sin^2 \pi \theta\right) (\zeta(2h+1, \theta) + \zeta(2h+1, 1-\theta))
\end{eqnarray}
 where $\zeta(.,.)$ is the classical Hurwitz zeta function. 

It is not difficult to see that the function $\varphi_h$ is continuous on $(0, 1)$ and satisfies 
    \begin{eqnarray*} \label{sddsdwdwwsas}
\varphi_h(t)  = O(t^{1-2h} (1-t)^{1-2h}),\,\, \mbox{ for } t \mbox{ near } 0 \mbox{ or } 1.
\end{eqnarray*}
This implies that both functions $\varphi(t)$ and $1/\varphi(t)$ are integrable on the unit circle. The inverse matrix $G^{-1}$ is therefore such that 
     $$\left(G^{-1}\right)_{k,j} =  \int_{\mathbb{T}} \frac{e^{-2\pi i (k-j) t}}{\varphi_h(t)} dt,\,\,\mbox{ for } j+ k \to \infty.$$
(See D'Ambrogi-Ola \cite{Dambrogi-ola}.)
This yields that $G^{-1}$ is asymptotically a Toeplitz matrix in the sense that for each $k,j$ fixed,
         $$\lim_{n\to \infty} \left(G^{-1}\right)_{k+n,j+n} = \int_{\mathbb{T}} \frac{e^{-2\pi i (k-j) t}}{\varphi_h(t)} dt = \widehat{\left(1/\varphi_h\right)} (k-j).$$
This implies in particular that $\left(G^{-1}\right)_{k,j} \to 0$ for $k+j \to \infty$ and hence 
$\sup_{k,j}|(G^{-1})_{k,j}| < \infty$ which guarantees that 
for each $z, w \in \mathbb{D}$ the series $Z^{T} G^{-1} \overline{W}$ converges (for $Z = (1, z, z^2, \ldots)$  and $W = (1, w, w^2, \ldots))$. 
The corresponding space $H^2_G(\mathbb{D})$ is the class of functions $ g \in H^2(\mathbb{D})$ such that
  $$\int_{\mathbb{T}} \left|g(e^{2 \pi i \theta})\right|^2 \varphi_h(\theta) d\theta < \infty.$$ 
For $h$ varying in $(0,1)$, this yields a family of sub-spaces of the Hardy space $H^2(\mathbb{D})$. 
 The exact entries of the matrix $G^{-1}$ are not known and therefore we do not have an explicit representation  of the kernel $\mathbb{K}_G(z, w)$ as  in the first two examples. However Theorem \ref{mainth} yields that if $(\xi_n)_{n\in \mathbb{N}}$ is a zero-mean complex Gaussian sequence of covariance matrix $G^{-1}$ and zero pseudo-covariance, then the zeros of $f(z) = \sum_{n =1}^\infty \xi_n z^{n-1}$ constitute a determinantal point process. 
 
As in the general case one can compute from the sequence of polynomials $1, z, z^2, \ldots$ a sequence of  orthonormal polynomials $\{P_n(z): n=1,2,\ldots\}$ and deduce that if $(\chi_n)$ is a sequence of i.i.d standard Gaussian random variables then the zeros of   $f(z) = \sum_{n=1}^\infty \chi_n P_n(z)$ constitute a determinantal point process. 
      
In the limit case where $h = 0$, the fractional Gaussian noise with index $h$ is such that the covariance matrix $G$ is given by 
      $$G_{k,k} = 1, G_{k,k+1} = G_{k+1, k} = -1/2 \mbox{ and } G_{k,j} = 0 \mbox{ for } |k-j| \geq 2,$$ and  it is not invertible. However it still determines a determinantal point process.  
The spectral density function of $G$ is 
           $$\varphi_0(\theta) = 1 - \cos(2\pi \theta),\,\, 0 \in \mathbb{T}.$$
 From the sequence of polynomials  $(1, z, z^2, z^3, \ldots)$, we derive the orthonormal sequence: 
  \begin{eqnarray*}
  P_n(z) = \left(\frac{2}{n(n+1)}\right)^{1/2}(1 + 2 z + 3 z^2 + \ldots+ n z^{n-1}),\,\,\,n = 1,2,\ldots
 \end{eqnarray*}
In this case, the kernel of $\mathbb{H}^2_G(\mathbb{D})$ is explicitly given by:
\begin{eqnarray*} \label{Saf_Kernel}
\mathbb{K}_0(z, w) = \sum_{n=1}^\infty P_n(z) \overline{P_n(w)} = \frac{2}{(1-z)(1-\overline{w})(1- z \overline{w})},\,\,\,z, w \in \mathbb{D}.
 \end{eqnarray*}
 This is exactly the limit case of the kernel given in Example \ref{exampl1} when the parameter $q$ approaches $ -1/2$. 
Then for a sequence $(\chi_n)$ of i.i.d standard Gaussian variables, the zeros of 
   \begin{eqnarray*}
   f(z) = \sum_{n=1}^\infty \chi_n P_n(z),\,\,\,\, z \in \mathbb{D}
   \end{eqnarray*}
constitute a determinantal point process. 



\begin{thebibliography}{100}

\bibitem{da Fonseca} da Fonseca, C.M. and  Petronilho, J. 2001.  Explicit inverses of some tridiagonal matrices
{\it Linear Algebra Appl.} 325, 7--21.

\bibitem{Dambrogi-ola} D'Ambrogi-Ola, B. 2009. {\it Inverse problem of fractional Brownian motion with discrete data}, PhD Thesis, University of Helsinki.

\bibitem{Escribano} Escribano, C., Gonzalo, R. and Torrano, E. 2015.  On the inversion of infinite moment matrices. {\it Linear Algebra Appl.} 475, 292--305.

\bibitem{Fikioris} Fikioris, G. 2018. Spectral properties of Kac--Murdock--Szeg\"o matrices with a complex parameter. {\it Linear Algebra Appl.} 553, 182--210


\bibitem{Hough} Hough, J.B.,  Krishnapur, M.,  
Peres, Y. and Vir\'ag, B. 2009. {\it Zeros of Gaussian analytic functions and determinantal point processes.} American Mathematical Society, Providence. 



\bibitem{Kahane}
Kahane, J.-P. 1985.
\newblock {\em Some random series of functions, 2nd ed.}
\newblock Cambridge University Press, Cambridge. 

\bibitem{Katznelson}
Katznelson, Y. 2004.
\newblock {\em An introduction to harmonic analysis , 3rd ed.}
\newblock Cambridge University Press, Cambridge. 


\bibitem{Krishnapur} Krishnapur, M.  2009. From random matrices to random analytic functions, {\it Ann. Probab.} 37(1),314--346.



\bibitem{Matsumoto} Matsumoto, S. and Shirai, T. 2013. Correlation functions for zeros of a Gaussian power series and Pfaffians. {\it Electron. J. Probab.} 18, no. 49, 1--18. 

\bibitem{Miller} Miller, K.S. 1969. Complex Gaussian processes. {\it SIAM Review} 4, 544--567.

\bibitem{Mukeru_2018} Mukeru, S. 2019.  Average number of real zeros of random algebraic
polynomials defined by the increments of fractional
Brownian motion. {\it J. Theor. Probab.} 32, 1502--1524.

\bibitem{Mukeru_Pi2} Mukeru, S.  A generalisation of Pisier homogeneous Banach algebra. {\it Michigan Math. J.}\\ To appear.

\bibitem{Mukeru_al} Mukeru, S., Mulaudzi, M.P., Nzabanita, J. and  Mpanda, M.M. 2020. Zeros of Gaussian power series with dependent random variables. {\it Illinois J. Math.} 64(4), 569--582.
 
\bibitem{Nourdin_Ivan} Nourdin, I. 2012. {\it Selected aspects of fractional Brownian motion.}  Bocconi University Press, Springer-Verlag. 

\bibitem{Paulsen} Paulsen, V.I. 2016. {\it An Introduction to the theory of reproducing kernel Hilbert spaces}, Cambridge University Press.

\bibitem{Peres_Virag} Peres, Y. and Vir\'ag, B. 2005. Zeros of the i.i.d. Gaussian power series: a conformally invariant determinantal process. {\it Acta Mathematica}, 194, 1--35.

 \bibitem{Sinai} Sinai,Y.G. 1976. Self-similar probability distributions. {\it Theory of Probability and its Applications}, 21, 64--80. 


\end{thebibliography}
\end{document}